\documentclass[12pt, letterpaper]{article}

\usepackage{fullpage}
\usepackage{amsfonts}
\usepackage{amsmath, amsthm}
\usepackage{amsthm}
\usepackage{amssymb}
\usepackage{cite}
 \usepackage{graphicx}
 \usepackage[all]{xy}

\newcommand{\F}{\mathbb{F}}

\newcommand{\Q}{\mathbb{Q}}

\newcommand{\q}{\mathfrak{q}}

\newtheorem{thm}{Theorem}
\newtheorem{lemma}{Lemma}

\newtheorem{corollary}{Corollary}
\newtheorem{remark}{Remark}

\begin{document}

\title{On $S_3$-extensions with infinite class field tower}
\author{Jonah Leshin}
\maketitle
\begin{abstract} We construct a class of $S_3$-extensions of $\Q$
  with infinite $3$-class field tower in which only three primes
  ramify. As an application, we obtain an $S_3$-extension of $\Q$
  with infinite $3$-class field tower with smallest known (to the author) root discriminant among all fields with infinite $3$-class field tower.

\end{abstract}

\section{Introduction}
Let $K:=K_0$ be a number field, and for $i \geq 1$, let $K_i$
denote the Hilbert class field of $K_{i-1}$-- that is, $K_i$ is the
maximum abelian unramified extension of $K_{i-1}$. The
tower \linebreak$K_0\subseteq K_1 \subseteq K_2\ldots$ is called the Hilbert class
field tower of $K$. If the tower stabilizes, meaning $K^i=K^{i+1}$ for
some $i$, then the class field tower is finite. Otherwise, $\cup_i K^i$ is an
infinite unramified extension of $K$, and $K$ is said to to have infinite class
field tower. For a prime $p$, we define the
$p$-Hilbert class field of $K$ to be the maximal abelian unramified
extension of $K$ of $p$-power degree over $K$. We may then analogously
define the $p$-Hilbert class field tower of $K$. In 1964, Golod and Shafarevich demonstrated the
  existence of a number field with infinite class field tower
  \cite{MR0218331}. This finding has motivated the construction of
  number fields with various properties that have infinite class field
  tower. One of Golod and
Shafarevich's examples of a number field with infinite class field
tower was any quadratic extension of the rationals ramified at
sufficiently many primes,
which was shown to have
infinite $2$-class field tower. An elementary exercise shows that if
$K$ has infinite class field tower, then any finite extension of $K$
does as well. Thus a task of interest becomes
finding number fields of small size with infinite class field towers. The size of a number field $K$ might be measured by the number of rational primes ramifying in $K$,
the size of the rational
primes ramifying in $K$, the root discriminant of $K$, or any combination of these three. 
\par With regard to number of primes ramifying, Schmithals \cite{Schmithals} gave
an example of a quadratic number field with infinite class field tower
in which a single rational prime
ramified. As for small primes ramifying, Hoelscher has given examples of number
fields with infinite class field tower ramified only at $p$ for $p=2,3$, and $5$ \cite{OnePrimeJing}
. Odlyzko's bounds \cite{OdlyzkoSummary} imply that any number field with infinite class field
tower must have root discriminant at least 22.3, (44.6 if
we assume GRH); Martinet showed that the number field $\Q(\zeta_{11}+\zeta_{11}^{-1}, \sqrt{46})$, with root
discriminant $\approx 92.4$, has infinite class field tower  \cite{Martinet}. 

\par In this note, we use a Theorem of Schoof to produce an infinite
class of $S_3$ extensions of $\Q$ with infinite class field tower. Our
fields are ramified at three primes, one of which is the prime 3. Our
main theorem is 

\begin{thm} \label{main}
Let $p\neq 3$ be prime and suppose the class number $h$ of $\Q(\omega,
\sqrt[3]{p})$ is at least 6, where $\omega$ is a primitive third root
of unity. For infinitely many primes $q$, there
exists $\delta \in \{p^aq^b\}_{1\leq a,b \leq 2}$ such that $\Q(\omega, \sqrt[3]{\delta})$ has infinite 3-class field tower. \end{thm}

As a direct consequence of the proof of Theorem \ref{main}, we find that $\Q\big(\omega, \sqrt[3]{79\cdot
  97}\big)$ has infinite 3-class field tower.

\section{Proof of Theorem \ref{main}}

\par Our construction is
analogous to that of Schoof \cite{Schoof}, Theorem 3.4.

\par We begin with a lemma.

\begin{lemma}\label{3lem}
Let $p$ be a rational prime different from $3$. The prime 3 ramifies totally
in $\Q(\sqrt[3]{p})$ if and only if $p \not \equiv \pm 1 \pmod{9}$.
\end{lemma}

\begin{proof}
Since $[\Q(\sqrt[3]{p}):\Q]$ is the sum of the local degrees
$[\Q(\sqrt[3]{p})_{\q_i}:\Q_3]$, where $\q_i$ is a prime of
$\Q(\sqrt[3]{p})$ above $3$, we see that 3 is totally ramified in
$\Q(\sqrt[3]{p})$ if and only if no third root of $p$ is contained in
$\Q_3$. Consider the equation $x^3-p \equiv 0 \pmod{27}$. This
equation has a solution if and only if $p \equiv \pm 1, \pm 8,$ or $\pm
10 \pmod{27}$. For such $p$, $3p^2$ is divisible by exactly one power
of $3$. Thus, by Hensel's lemma, we conclude that $\Q_3$ contains a
cube root of $p$ exactly when $p \equiv \pm 1, \pm 8,$ or $ \pm 10
\pmod{27}$. But these congruences are equivalent to the congruence
$p \equiv \pm 1 \pmod{9}$. 
\end{proof}

\begin{remark} \label{3rk}
The same proof shows that Lemma \ref{3lem} holds if $p$ is replaced by
any integer that is prime to 3 and not a perfect cube. 
\end{remark}

\par  Let $p$ be any prime different from $3$, and let $h$ be the
class number of $\Q(\omega, \sqrt[3]{p})$ with $H$ its Hilbert class field. Let $q$ be a rational prime that splits
completely in $H$, so by
class field theory, $q$ is a prime that splits completely into
principal prime ideals in
$\Q(\omega, \sqrt[3]{p}$). In what follows, we find $\delta=\delta_{p,q}
\in \{p^aq^b\}_{1\leq a, b \leq 2}$ so that $E$ is unramified over
$K:=F(\sqrt[3]{\delta})$ (see Figure \ref{fig}). 

\par Suppose that $p \not \equiv \pm 1
\pmod{9}$, so that 3 ramifies completely in $\Q(\sqrt[3]{p})$ by Lemma
\ref{3lem}. Then either $pq \not \equiv \pm 1 \pmod{9}$ or $pq^2 \not \equiv
\pm 1 \pmod{9}$ (or both). Pick $\delta=pq$ or $\delta=pq^2$ so that
$\delta \not \equiv \pm 1 \pmod{9}$. Let $F=\Q(\omega)$ and
$E=F(\sqrt[3]{p}, \sqrt[3]{q})$. The ramification degree $e(F(\sqrt[3]{p}),3)$ of 3 in
$F(\sqrt[3]{p})$ is 6. 
\par We claim that $e(E,3)=6$. Suppose for contradiction that this is
not so, in which case we must have $e(E,3)=18$. This means that the field
$E$ has a single prime $\mathfrak{l}$ lying above 3,
and that $[E_{\mathfrak{l}}:\Q_3]=[E:\Q]$, and likewise for
every intermediate field between $\Q$ and $E$. One checks that at least one element of the set $\{q, pq, pq^2
\pmod{9}\}$ is congruent to $\pm 1 \pmod{9}$. Pick $\gamma \in \{q, pq,
pq^2\}$ so that $\gamma \equiv \pm 1 \pmod{9}$. Let $E'=\Q(\sqrt[3]{\gamma})$.
The extension $E'/\Q$ is degree 3, but the corresponding extension of
local fields is either degree one or two (depending on which prime
above three in $E'$ one chooses). This gives the desired
contradiction. 
\par We claim that $E/K$ is unramified. Since $E$ is generated over
$K$ by either $x^3-p$ or $x^3-q$, the relative discriminant of $E/K$ must divide the ideal $(3^3)$ of $K$. Therefore, the only
possible primes of $K$ that can ramify in $E$ are those lying above
3. It is necessary and sufficient to show that $e(K,3)=6$. Since
$\delta \not \equiv \pm 1 \pmod{9}$, Remark \ref{3rk} shows that 3 is
totally ramified in $\Q(\sqrt[3]{\delta})$, from which it follows that
$e(K,3)=6$. 
\par Suppose now that $p \equiv \pm 1 \pmod{9}$. If $q \not \equiv \pm
1 \pmod{9}$, take $\delta=pq$. The previous argument, with the roles
of $p$ and $q$ now reversed, shows that $e(E,3)=6$ and that $E/K$ is
unramified. If $q \equiv \pm 1 \pmod{9}$, then
there is a prime of $\Q(\sqrt[3]{q})$ lying above 3 and unramified
over $3$, and likewise for $\Q(\sqrt[3]{p})$. It follows that
$\Q(\sqrt[3]{p}, \sqrt[3]{q})$ also has a prime lying above 3 and
unramified over 3, and from here that $e(E,\Q)=2$. So in the case $p,q
\equiv \pm{1} \pmod{9}$, we may take $\delta$ to be any element of
$\{p^aq^b\}_{1 \leq a, b \leq 2}$, and $E/K$ will be unramified.

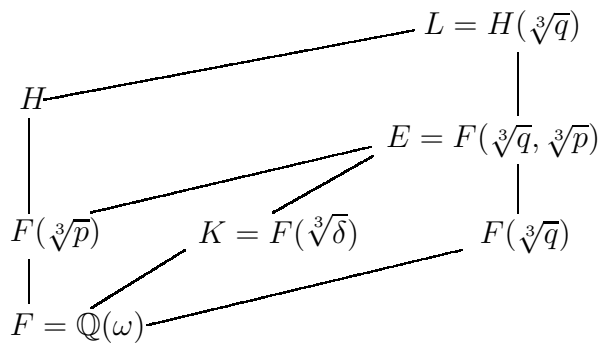
\begin{figure}
\setlength{\unitlength}{2.5cm}
\caption{Field Diagram for Theorem \ref{main}.}\label{fig}
\begin{picture}(2.5,2.5)(-1.6,-.4)
\put(0.0, 0.0){$F=\Q(\omega)$}
\put(0.0, .5){$F(\sqrt[3]{p})$}
\put(1,.49){$K=F(\sqrt[3]{\delta})$}
\put(2,1.0){$E=F(\sqrt[3]{q}, \sqrt[3]{p})$}
\put(2.5,.49){$F(\sqrt[3]{q})$}
\put(.05,1.2){$H$}
\put(2.2,1.6){$L=H(\sqrt[3]{q})$}
\qbezier(.1,.17)(.1,.17)(.1,.4)
\qbezier(.43,.14)(.43,.14)(.93,.45)
\qbezier(.73,.05)(.73,.05)(2.4,.45)
\qbezier(2.7,.65)(2.7,.65)(2.7,.9)
\qbezier(1.4,.65)(1.4,.65)(1.93,.95)
\qbezier(.43,.65)(.43,.65)(1.92,1)
\qbezier(.1,.65)(.1,.65)(.1,1.15)
\qbezier(2.7,1.17)(2.7,1.17)(2.7,1.5)
\qbezier(.18, 1.25)(.18, 1.25)(2.15,1.62)

\end{picture}
\end{figure}

\par We are now ready to invoke the theorem of Schoof \cite{Schoof}. First
we set notation. Given any number field $H$, let $O_H$ denote the
ring of integers of $H$. Let $U_H$ be the
units in the id\`ele group of $H$-- that is, the id\`eles with valuation
zero at all finite places. Given a finite extension $L$ of $H$, we
have the norm map $N_{U_L/U_H}:U_L \to U_H$, which is just the
restriction of the norm map from the id\`eles of $L$ to the id\`eles
of $H$. We may view $O_H^*$ as a
subgroup of $U_H$ by embedding it along the diagonal. Given a finitely generated abelian group
$A$, let $d_l(A)$ denote the dimension of the $\F_l$-vector space
$A/lA$. \\

\begin{thm}\textup{[Schoof]\cite{Schoof}}
 Let $H$ be a number field. Let $L/H$ be a cyclic extension of prime
 degree $l$, and
 let $\rho$ denote the number of primes (both finite and infinite) of
 $H$ that ramify
in $L$. Then $L$ has infinite $l$-class field tower if 
\begin{align*}
\rho\geq 3+d_l\big(O_H^*/(O_H^*\cap N_{U_L/U_H}U_L)\big)+2\sqrt{d_l(O_L^*)+1}\hspace{.1cm} .
\end{align*}
\end{thm}
\par We apply Schoof's theorem to the extension $L:=H(\sqrt[3]{q})$ over $H$,
where $H$, as above, is the Hilbert class field of $F(\sqrt[3]{p})$. All $6h$ primes in $H$ above $q$ ramify completely in the field
$H(\sqrt[3]{q})$. Thus $\rho \geq 6h$, with strict inequality if and
only if the primes above 3 in $H$ ramify. By Dirichlet's unit theorem,
$d_3(O_L^*)=9h$ and $d_3(O_H^*)=3h$. Thus if $h$ satisfies $6h \geq
3+3h+2\sqrt{9h+1}$, then $L$ will have infinite $3$-class field
tower. Since $L/K$ is an unramified (as both $L/E$ and $E/K$ are
unramified) solvable extension, it follows that $K$ has infinite
class field tower as well. The minimal such $h$ is given by
$h=6$.
\par This proves the following version of our main theorem:

\begin{thm}\label{main2}
Let $p\neq 3$ be prime and suppose the class number $h$ of $\Q(\omega, \sqrt[3]{p})$ is at least 6. Let $q$ be
a prime that splits completely into principal ideals in $\Q(\omega,
\sqrt[3]{p})$. Then there exists $\delta \in \{p^aq^b\}_{1 \leq a,b
  \leq 2}$ such that $\Q(\omega, \sqrt[3]{\delta})$ has infinite class field tower. 
\end{thm}

\begin{remark}
By the Chebotarev density theorem, the density of such $q$ is
$\frac{1}{6h}$. 
\end{remark}

\begin{remark}
Since $\delta \equiv \pm 1 \pmod{9}$ if and only if $\delta^2 \equiv
\pm 1 \pmod{9}$, the proof of Theorem \ref{main2} goes through with
$\delta$ replaced by $\delta^2$. Thus we always generate at least two
$S_3$ extensions of $\Q$ unramified outside $\{3, p, q\}$ with infinite class
field tower. 
\end{remark}

\par The field $\Q(\omega, \sqrt[3]{79})$ has class number 12, and
$97$ splits completely into a product of principal ideals in this
field \cite{sage}, so
we obtain

\begin{corollary}
The field $\Q\big(\omega, \sqrt[3]{79\cdot
  97}\big)$ has infinite 3-class field tower.
\end{corollary}

\begin{remark}
It is a Theorem of Koch and Venkov \cite{kv} that a quadratic
imaginary field whose class group has $p$-rank three or larger has
infinite $p$-class field tower. From the tables in
\cite{tables}, we see that the smallest known imaginary quadratic field with
$3$-rank at least three is $\Q(\sqrt{-3321607})$, with root
discriminant $\approx 1822.5$. The field $\Q\big(\omega, \sqrt[3]{79\cdot
  97}\big)$ has root discriminant $\approx 1400.4$. 
\end{remark}

\par We can bring down the requirement $h \geq 6$. The trade off is
that we will have to assume 
\begin{equation}\label{*}
\textrm{ The primes of } H \textrm{ that ramify in }L  \textrm{ split completely in } H(\sqrt[3]{O_H^*}).
  \end{equation}

\par If $p \equiv \pm 1 \pmod{9}$ and $q \not \equiv \pm 1 \pmod{9}$,
then ramification considerations show that the primes above $3$ in $H$ ramify in $L$; otherwise, the only
primes in $H$ ramifying in $L$ are those above $q$.

\par Suppose \eqref{*} holds. We claim that $O_H^*\cap N_{U_L/U_H}U_L=O_H^*$. Let $x$ be
an arbitrary element of $O_H^*$. We construct $y=(y_w) \in U_L$ such that
$Ny=x$.  Consider first the primes of $H$ that are
unramified in $L$. Let $v$ be such a prime and suppose $w_1, \ldots,
w_a$ $(a=1 \textrm{ or } 3)$ are the primes above $v$ in $L$. Because $v$ is unramified, the local norm
map $N:O_{L_{w_i}}^*\to O_{H_v}^*$ is surjective, so we can pick $y_v
\in L_{w_1}$ such that $Ny_v=x$. Put $1$ in the $w_2$ and $w_3$ components of $y$ if
$a=3$.
\par Now let $v$ be a prime of $H$ that ramifies in $L$. The assumption that $v$ splits completely in
$H(\sqrt[3]{O_H^*})$ means that $\sqrt[3]{O_H^*} \in H_v$. Letting $w_1, w_2, w_3$ be the primes
above $v$, we can put $\sqrt[3]{x}$ in
the $w_1$ component and $1$ in the $w_2$ and $w_3$ components of
$y$. Putting the ramified and unramified components of $y$ together
gives the desired element.
\par Under assumption \eqref{*}, the inequality needed for $L$ to have an infinite class field tower now becomes
\begin{align*}
6h \geq 3+2\sqrt{9h+1},
\end{align*}
 which is satisfied by $h \geq 2$. This gives

\begin{thm}
Let $p$ be a prime with $p \not \equiv \pm 1 \pmod{9}$. Either $\Q(\omega,
\sqrt[3]{p})$ has class number one, or there exist infinitely many primes $\{q_p\}$ such that $\Q(\omega,
\sqrt[3]{\delta_{p,q_p}})$ has infinite class field tower. 
\end{thm}

\begin{proof}
  For such $p$, the set $\{q_p\}$ consists of all rational primes
  splitting completely in $H(\sqrt[3]{O_H^*})$.
\end{proof}

\bibliographystyle{plain}
\bibliography{myrefs}
\end{document}